\documentclass[12pt,a4paper]{amsart}
 \usepackage{amsfonts,amsmath}
 \usepackage{amssymb}
 \usepackage{color}
\numberwithin{equation}{section}

\theoremstyle{plain}
\newtheorem{Th}{Theorem}[section]
\newtheorem{Lemma}[Th]{Lemma}
\newtheorem{Prop}[Th]{Proposition}
\newtheorem{Cor}[Th]{Corollary}

\theoremstyle{definition}
\newtheorem{Def}[Th]{Definition}
\newtheorem{Note}[Th]{Note}

\newtheorem{Rem}[Th]{Remark}

  \newcommand{\ab}[1]{\left\vert{#1}\right\vert}
 
 \newcommand{\low}[1]{\left\lfloor{#1}\right\rfloor}
 \newcommand{\up}[1]{\left\lceil {#1}\right\rceil }

  \newcommand{\ce}{{\subseteq}}
  \newcommand{\aaa}{{\underline{a}}}
  \newcommand{\PP}{{\mathbb{P}}}
  \newcommand{\E}{{\mathbb{E}}}
  \newcommand{\I}{\mathbb I}
\begin{document}

\title{Generalized Sidon sets}

\author{Javier Cilleruelo}
\address{Departamento de Matem\'aticas. Universidad Aut\'onoma de Madrid, 28049 - Madrid, Spain.}
\thanks{During the preparation of this paper, J. C. and C. V. were supported  by Grant
MTM 2008-03880 of MYCIT} \email{franciscojavier.cilleruelo@uam.es}

\author{Imre Z. Ruzsa}
\address{Alfr\'ed R\'enyi Institute of Mathematics\\
     Budapest, Pf. 127\\
     H-1364 Hungary.}
\email{ruzsa@renyi.hu}

\thanks{I. R. is supported by Hungarian National Foundation for Scientific Research
(OTKA), Grants No. K 61908, K 72731.}

\author{Carlos Vinuesa}
\address{Departamento de Matem\'aticas. Universidad Aut\'onoma de Madrid, 28049 - Madrid, Spain.}
\thanks{C. V. would like to thank I. R. for his hospitality,
as well as that of the Alfr\'ed R\'enyi Institute of Mathematics,
during his stay in Budapest.} \email{c.vinuesa@uam.es}

    \begin{abstract}
    We give asymptotic sharp estimates for the cardinality of a set of residue
classes with the property that the representation function is
bounded by a prescribed number. We then use this to obtain an
analogous result for sets of integers, answering an old question of
Simon Sidon.
    \end{abstract}

     \maketitle

     \section{Introduction}

     A \emph{Sidon set} $A$ in a commutative group is a set with the property
that the sums $a_1+a_2$, $a_i\in A$ are all distinct except when
they coincide because of commutativity. We consider the case when,
instead of that, a bound is imposed on the number of such
representations. When this bound is $g$, these sets are often called
$B_2[g]$ sets. This being both clumsy and ambiguous, we will avoid
it, and fix our notation and terminology below.

     Our main interest is in sets of integers and residue classes, but we
formulate our concepts and some results in a more general setting.

     Let $G$ be a commutative group.

     \begin{Def} \label{defrep}
For $A \subset  G$, we define the corresponding \emph{representation function} as
     $$r(x) = \sharp \{(a_1,a_2): a_i \in A, \, a_1 + a_2 = x\}.$$
The \emph{restricted representation function} is
$$r'(x) = \sharp \{(a_1,a_2): a_i \in A,\, a_1 + a_2 = x,\, a_1\ne a_2\}.$$
Finally, the \emph{unordered representation function} $r^*(x)$ counts the pairs
$(a_1, a_2)$ where
$(a_1, a_2)$  and $(a_2, a_1)$ are identified. With an ordering given on $G$
(not necessarily in any connection with the group operation) we can write this
as
$$r^*(x) = \sharp \{(a_1,a_2): a_i \in A,\, a_1 + a_2 = x, \,a_1\leq a_2\}.$$
     \end{Def}

     These functions are not independent; we have always the equality
     \[   r^*(x) = r(x) - \frac{r'(x)}{2}   \]
     and the inequalities
     \[   r'(x) \leq  r(x) \leq  2r^*(x) . \]
     We have $r(x)=r'(x)$ except for $x=2a$ with $a\in A$. If we are in this last case and there are no
elements of order 2 in $G$, then necessarily $r(x)=r'(x)+1$, and
the quantities are more closely connected:
     \[   r'(x) = 2 \low{\frac{r(x)}{2}}, \ r^*(x) = \up{\frac{r(x)}{2}} .  \]
     This is the case in ${\mathbb {Z}}$, or in ${\mathbb {Z}}_q$ for odd values of $q$. For even $q$
this is not necessarily true, but both for constructions and estimates the
difference seems to be negligible, as we shall see. In a group with lots of
elements of order 2, like in ${\mathbb {Z}}_2^m$, the difference is substantial.

     Observe that $r$ and $r'$ make sense in a noncommutative group as well,
while $r^*$ does not.

     \begin{Def} 
     We say that $A$ is a \emph{$g$-Sidon set}, if $r(x)\leq g$ for all $x$.
It is a \emph{weak $g$-Sidon set}, if $r'(x)\leq g$ for all $x$.
It is an \emph{unordered $g$-Sidon set}, if $r^*(x)\leq g$ for all $x$.
     \end{Def}

\begin{Note}
When we have a set of integers $C \  \ce \ [1, m]$, we say that it is a $g$-Sidon set $\pmod m$ if the residue classes $\{c \pmod{m} : c \in C\}$ form a $g$-Sidon set in $\mathbb{Z}_m$.
\end{Note}

     The strongest possible of these concepts is that of an unordered 1-Sidon
set, and this is what is generally simply called a Sidon set. A weak 2-Sidon
set is sometimes called a weak Sidon set.

     These concepts are closely connected. If there are no elements of order 2,
then  $2k$-Sidon sets and unordered $k$-Sidon sets coincide, in
particular, a Sidon set is the same as a 2-Sidon set. Also, in this
case $(2k+1)$-Sidon sets and weak $2k$-Sidon sets coincide.
Specially, a 3-Sidon set and a weak 2-Sidon set are the same.

Our aim is to find estimates for the maximal size of a $g$-Sidon set
in a finite group, or in an interval of integers.

\subsection{The origin of the problem: g-Sidon sets in the integers}
In 1932, the analyst S. Sidon asked to a young P. Erd\H os about the
maximal cardinality of a $g$-Sidon set of integers in $\{1,\dots , n\}$.
Sidon was interested in this problem in connection
with the study of the $L_p$ norm of Fourier series with frequencies
in these sets but Erd\H os was captivated by the combinatorial and
arithmetical flavour of this problem and it was one of his favorite
problems; not in vain it has been one of the main topics in
Combinatorial Number Theory.


  \begin{Def} 
     For a positive integer $n$
     \[   \beta _g(n) = \max \ab A : A \subset \{1, \dots , n\}, \, A \text{ is a }
g\text{-Sidon set} . \]
     We define $\beta '_g(n)$ and $\beta ^*_g(n)$ analogously.
     \end{Def}

 The behaviour of this quantity is  only known for classical
     Sidon sets and for weak Sidon
sets : we have $\beta _2(n)\sim \sqrt n$ and $\beta _3(n)\sim \sqrt
n$.

The reason which makes easier the case $g=2$ is that
     2-Sidon sets have the property that the differences $a-a'$
     are all distinct. Erd\H os an Tur\'an \cite{ET} used this to prove that
     $\beta_2(n)\le \sqrt n+O(n^{1/4})$ and Lindstr\"om \cite{L} refined that
     to get $\beta_2(n)\le \sqrt n+n^{1/4}+1$. For weak Sidon sets  Ruzsa \cite{Ru} proved that
     $\beta_3(n)\le \sqrt n+4n^{1/4}+11$.

     For the lower bounds, the classical constructions of Sidon sets of Singer \cite{Si}, Bose \cite{BC} and Ruzsa \cite{Ru} in some finite groups, $\mathbb Z_m$,
          give $\beta_3(n)\ge \beta_2(n)\ge \sqrt n(1+o(1))$. Then,
      $\lim_{n\to \infty}\frac{\beta_2(n)}{\sqrt n}=\lim_{n\to \infty}\frac{\beta_3(n)}{\sqrt
      n}=1$.

However for $g\ge 4$ it has not even been proved that $\lim_{n\to
\infty}\beta_g(n)/\sqrt n$ exists.

For this reason we write
     \[  \overline \beta _g = \limsup_{n \to \infty} \beta _g(n)/\sqrt n \qquad \text{and} \qquad \underline \beta _g = \liminf_{n \to \infty} \beta _g(n)/\sqrt n . \]
     It is very likely that these limits coincide, but this has only been proved for
     $g=2,3$.
     A wide literature has been written with bounds for
     $\overline{\beta}_g$ and $\underline{\beta}_g$ for arbitrary $g$. The trivial
     counting argument gives $\overline{\beta}_g\le \sqrt{2g}$
     while the strategy of pasting Sidon sets in $\mathbb Z_m$ in the obvious way gives
     $\underline{\beta}_g\ge \sqrt{g/2}$.

     The problem of narrowing this gap has attracted the attention of many mathematicians
     in the last years.

     For example, while for $g=4$ the trivial upper bound gives
     $\overline{\beta}_4\le \sqrt 8$, it was proved in \cite{Ci1} that
     $\overline{\beta}_4\le \sqrt 6$, which was refined to
     $\overline{\beta}_4\le 2.3635...$ in \cite{P} and to
     $\overline{\beta}_4\le 2.3218...$ in \cite{HP}.

     On the other hand, Kolountzakis \cite{K} proved that $\underline{\beta}_4\ge \sqrt 2$, which  was improved to $\underline{\beta}_4\ge 3/2$ in \cite{CRT} and to $\underline{\beta}_4\ge  4/\sqrt
     7=1.5118...$ in \cite{HP}.

     We describe below the progress done for large $g$:

     \smallskip

\begin{center}\begin{tabular}{ll}
$\frac{\overline{\beta}_g}{\sqrt g}$ & $\le \sqrt 2 = 1.4142...$ (trivial) \\
& $\le 1.3180... $ (J. Cilleruelo - I. Z. Ruzsa - C. Trujillo,\ \cite{CRT})\\
& $\le 1.3039... $ (B. Green,\ \cite{G}) \\
& $\le 1.3003... $ (G. Martin - K. O'Bryant,\ \cite{Mar1})\\
& $\le 1.2649... $ (G. Yu,\ \cite{Yu})\\
& $\le 1.2588... $ (G. Martin - K. O'Bryant,\ \cite{Mar2})
\end{tabular}\end{center}

\smallskip

\begin{center}\begin{tabular}{ll}
$\lim_{g\to \infty}\frac{\underline{\beta}_g}{\sqrt g}$ & $\ge 1/\sqrt 2 = 0.7071...$ (M. Kolountzakis,\ \cite{K}) \\
& $\ge 0.75 $ (J. Cilleruelo - I. Z. Ruzsa - C. Trujillo,\ \cite{CRT})\\
& $\ge 0.7933... $ (G. Martin - K. O'Bryant,\ \cite{Mar})\\
& $\ge \sqrt{2/\pi}=0.7978... $ (J. Cilleruelo - C. Vinuesa,\ \cite{CV}).
\end{tabular}\end{center}

\bigskip

 Our main result connects this problem with a quantity arising from the analogous
continuous problem, first studied by Schinzel and Schmidt
\cite{Schin}. Consider all nonnegative real functions $f$ satisfying
$f(x)=0$ for all $x\notin [0,1]$, and
     \[   \int _0^1 f(t) f(x-t) \, dt \leq  1  \]
     for all $x$.  Define the constant $\sigma $ by
     \begin{equation} \label{sigma}
\sigma  = \sup \int _0^1 f(x) \, dx
     \end{equation}
     where the supremum is taken over all functions $f$ satisfying the above
restrictions.

     \begin{Th}  \label{integer}
     \[ \lim_{g\to \infty}  \frac{\underline \beta _g}{\sqrt g} =
     \lim_{g\to \infty}  \frac{\overline \beta _g}{\sqrt g} = \sigma  . \]
     \end{Th}

In other words, the theorem above says that the maximal cardinality of a $g$-Sidon set in $\{1, \dots, n\}$ is
$$\beta_g(n) = \sigma \sqrt{gn}(1-\varepsilon(g,n))$$
where $\varepsilon(g,n) \to 0$ when both $g$ and $n$ go to infinity.

Schinzel and Schmidt \cite{Schin} and Martin and O'Bryant
\cite{Mar2} conjectured that $\sigma =2/\sqrt \pi = 1.1283...$, and an extremal
function was given by $f(x)=1/\sqrt {\pi x}$ for $0<x\leq 1$. But
recently this has been disproved \cite{MV} with an explicit $f$
which gives a greater value. The current state of the art for this
constant is
$$1.1509... \le \sigma\le 1.2525...$$
both bounds coming from \cite{MV}.

The main difficulty in Theorem \ref{integer} is establishing the lower
bound for $\lim \frac{\underline \beta _g}{\sqrt g}$. Indeed the
upper bound $\lim \frac{\overline \beta _g}{\sqrt g} \le \sigma$
was already proved in \cite{CV} using a result of Schinzel and Schmidt from \cite{Schin}.
We include however a complete proof of the theorem.

The usual strategy to construct large $g$-Sidon sets in the integers
is pasting large Sidon sets modulo $m$ in a suitable form. The strategy
of pasting $g$-Sidon sets modulo $m$ had not been tried before since there
were no large enough known $g$-Sidon sets modulo $m$.

Precisely, the heart of the proof of this theorem is the construction of large $g$-Sidon sets modulo $m$.

\subsection{g-Sidon sets in finite groups}

     \begin{Def} 
     For a finite commutative group $G$ write
     \[   \alpha _g(G) = \max \ab A : A \subset G, \, A \text{ is a } g\text{-Sidon set} . \]
     We define $\alpha '_g(G)$ and $\alpha ^*_g(G)$ analogously.
     For the cyclic group $G={\mathbb {Z}}_q$, with an abuse of notation, we write
$\alpha _g(q) = \alpha _g({\mathbb {Z}}_q)$.
     \end{Def}

     An obvious estimate of this quantity is
     \[   \alpha _g(q) \leq  \sqrt {gq} .  \]
     Our aim is to show that for large $g$ for some values of $q$ this is
asymptotically the correct value. More exactly, write
     \[   \alpha _g = \limsup_{q \to \infty} \alpha _g(q)/\sqrt q .  \]
The case $g=2$ (Sidon sets) is well known, we have $\alpha _2 = 1$.
It is also known \cite{Ru} that $\alpha_3 =1$. Very little is known
about $\alpha_g$ for $g \ge 4$.

For $g = 2k^2$, Martin and O'Bryant \cite{Mar} generalized the well
known constructions of Singer \cite{Si}, Bose \cite{BC} and Ruzsa
\cite{Ru}, obtaining $\alpha_g \ge \sqrt{g/2}$ for these values of
$g$.

We are unable to exactly determine $\alpha _g$ for any  $g\geq 4$, but we will find its
asymptotic behaviour. Our main result sounds as follows.

     \begin{Th}  \label{modular}
     We have
     \[   \alpha _g = \sqrt g + O\left( g^{3/10} \right),  \]
     in particular,
     \[ \lim_{g\to \infty}  \frac{\alpha _g}{\sqrt g} = 1.  \]
     \end{Th}

\bigskip

     In Section 2, as a warm-up, we give a slight improvement of the obvious upper estimate.

In Section 3 we construct dense $g$-Sidon sets in groups ${\mathbb
{Z}}_p^2$. In Section 4 we use this to construct $g$-Sidon sets
modulo $q$ for certain values of $q$.

Section 5 is devoted to the proof of the upper bound of Theorem
\ref{integer}. In Section 6 we prove the lower bound of Theorem \ref{integer}
pasting copies of the large g-Sidon sets in $\mathbb Z_q$ which we
constructed in Section 4. In these two sections, we connect the discrete
and the continuous world, combining some ideas from Schinzel and Schmidt
and some probabilistic arguments used in \cite{CV}.

     \section{An upper estimate}

     The representation function $r(x)$ behaves differently
at elements of $2\cdot A=\{2a:a\in A\}$ and the rest; in particular, it can be odd only
on this set. Hence we formulate our result in a flexible
form that takes this into account.

     \begin{Th} 
     Let $G$ be a finite commutative group with $\ab G =q$. Let $k\geq 2$ and $l\geq 0$ be
integers and $A\subset G$ a set such that the corresponding representation function
satisfies
     \[    r(x) \leq  \left\{ \begin{array}{lcc}
k, & \text{ if } x \notin  2\cdot A, \\
k+l, & \text{ if }  x \in  2\cdot A . \\
\end{array} \right.        \]
     We have
     \begin{equation} \label{Aup}
     \ab A < \sqrt {(k-1)q} + 1 + \frac{l}{2} + \frac{l(l+1)}{2(k-1)} . \end{equation}
     \end{Th}

     \begin{Cor} 
     Let $G$ be a finite commutative group with $\ab G =q$, and let
$A\subset G$ be a $g$-Sidon set. If $g$ is even, then
     \[   \ab A \leq \sqrt {(g-1)q} + 1  .  \]
     If $g$ is odd, then
     \[   \ab A \leq \sqrt {(g-2)q} + \frac{3}{2} + \frac{1}{g-2}  .  \]
     \end{Cor}

     Indeed, these are cases $k=g$, $l=0$ and $k=g-1$, $l=1$ of the previous
theorem.

     \begin{Cor} 
     Let $A \subset {\mathbb {Z}}_q$ be a weak $g$-Sidon set. If $q$ is even, then
     \[   \ab A \leq \sqrt {(g-1)q} + 2 + \frac{3}{g-1}  .  \]
     If $q$ is odd, then
     \[   \ab A \leq \sqrt {(g-1)q} + \frac{3}{2} + \frac{1}{g-1}  .  \]
     \end{Cor}

     To deduce this, we put $k=g$ and $l=2$ if $q$ is even, $l=1$ if $q$ is
odd.

     \begin{proof} 
     Write $\ab A = m$.
     We shall estimate the quantity
     \[   R = \sum  r(x)^2  \]
     in two ways.

     First, observe that
     \[    r(x)^2 - kr(x) = r(x) \left( r(x)-k\right) \leq  \left\{ \begin{array}{lcc}
0, & \text{ if } x \notin  2\cdot A, \\
l(k+l), & \text{ if }  x \in  2\cdot A , \\
\end{array} \right.        \]
     hence
     \[   R \leq  k \sum  r(x) + l(k+l) \ab{2\cdot A} .  \]
     Since clearly $\sum  r(x)= m^2$ and $ \ab {2\cdot A}\leq m$, we conclude
     \begin{equation} \label{Rup}   R \leq  km^2 + l(k+l)m .  \end{equation}

     Write
     $$ d(x) = \sharp \{(a_1,a_2): a_i \in A, \, a_1 - a_2 = x\}.$$
     Clearly $d(0)=m$. We also have $\sum  d(x) = m^2$, and, since the equations
$x+y=u+v$ and $x-u=v-y$ are equivalent,
     \[   \sum  d(x)^2 = R .  \]
     We separate the contribution of $x=0$ and use the inequality of the arithmetic
and quadratic mean to conclude
     \[   R = m^2 + \sum _{x\ne 0} d(x)^2 \geq  m^2 + \frac{1}{q-1} \left( \sum _{x\ne 0} d(x)
\right)^2 > m^2 + \frac{m^2 (m-1)^2}{q} . \]
     A comparison with the upper estimate \eqref  {Rup} yields
     \[  \frac{m^2 (m-1)^2}{q} <   (k-1)m^2 + l(k+l)m .  \]
     This can be rearranged as
     \[   (m-1)^2 < (k-1)q + \frac{l(k+l)q}{m} . \]
     Now if $m<\sqrt {(k-1)q}$, then we are done; if not, we use the opposite
inequality to estimate the second summand and we get
     \[   (m-1)^2 < (k-1)q + \frac{l(k+l)\sqrt q}{\sqrt {k-1}} . \]
     We take square root and use the inequality $\sqrt {x+y}\leq  \sqrt x + \frac{y}{2\sqrt x}$ to
obtain
     \[   m-1 < \sqrt {(k-1)q} + \frac{l(k+l)}{2(k-1)} \]
     which can be written as \eqref  {Aup}.
     \end{proof}

     \section{Construction in certain groups}

     In this section we construct large $g$-Sidon sets in groups $G={\mathbb {Z}}_p^2$, for
primes $p$. We shall establish the following result.

\begin{Th} \label{pxp}
Given $k$,  for every sufficiently large prime $p \geq p_0(k)$ there is
a set $A \ce {\mathbb {Z}}_p^2$ with $kp - k + 1$ elements which is a $g$-Sidon
set for $g= \lfloor k^2 + 2 k^{3/2} \rfloor$.
\end{Th}

Observe that the trivial upper bound in this case is
    $$|A| \leq \sqrt{gq} \le kp \sqrt{1 + \frac{2}{\sqrt{k}}} < (k+\sqrt k)p . $$

     \begin{proof} 
Let $p$ be a prime.
For every $u \not \equiv 0$ in ${\mathbb {Z}}_p$ consider the set
    $$A_u = \left\{ \left( x, \frac{x^2}{u} \right) : x \in {\mathbb {Z}}_p \right\} \subset  {\mathbb {Z}}_p^2.$$
    Clearly  $|A_u| = p$.

     We are going to study the sumset of two such sets.
For any $\aaa = (a, b) \in {\mathbb {Z}}_p^2$ we shall calculate the representation function
$$r_{u,v}(\aaa) = \sharp \{(\aaa_1, \aaa_2): \aaa_1 \in A_u, \aaa_2 \in A_v, \aaa_1 + \aaa_2 = \aaa \}.$$
The most important property for us sounds as follows.

\begin{Lemma} \label{u+v}
If $u + v \equiv u' + v'$ and $\left(\frac{uvu'v'}{p} \right) = -1$ then $r_{u,v}(x) + r_{u',v'}(x) = 2$ for all $x$.
\end{Lemma}

     \begin{proof} 

If $a \equiv x + y$ and $b \equiv \frac{x^2}{u} + \frac{y^2}{v}$,
with $uv \not \equiv 0$, then $y \equiv a- x$ and we have $b \equiv
\frac{x^2}{u} + \frac{(a - x)^2}{v}$. We can rewrite this equation
as $(u + v)x^2 - 2aux + ua^2 - buv \equiv 0$. The discriminant of
this quadratic equation is
 $\Delta \equiv 4uv((u+v)b - a^2)$. The number of solutions is

\begin{equation*}
r_{u,v}(a,b) =
\left\{
\begin{array}{lccrl}
1 \qquad \text{ if } & \left (\frac{\Delta}p\right )  & = &0 & \\
2 \qquad \text{ if } &\left( \frac{\Delta}{p} \right) &=& + 1&\ \text{(}
\Delta \text{ quadratic residue)}\\
0 \qquad \text{ if } &\left( \frac{\Delta}{p} \right) &=& -1& \ \text{(}
\Delta \text{ quadratic nonresidue)}.\\ \end{array} \right.
\end{equation*}

     We can express this as
\begin{equation*}
r_{u,v}(a,b) = 1 +  \left( \frac{\Delta}{p} \right) .
\end{equation*}

\smallskip

Now, since
$$\Delta\Delta'\equiv 4uv((u+v)b-a^2)4u'v'((u'+v')b-a^2)\equiv
16uvu'v'((u+v)b-a^2)^2$$ we have
$$\left (\frac{\Delta}p\right )\left (\frac{\Delta'}p\right )=\left (\frac{\Delta \Delta'}p\right )=\left (\frac{uvu'v'}p\right
)\left (\frac{((u+v)b-a^2)^2}p\right )=-\left
(\frac{((u+v)b-a^2)^2}p\right ).$$ If $(u+v)b-a^2\equiv 0$, we have
$\left (\frac{\Delta}p\right )=\left (\frac{\Delta'}p\right )=0$. If
not, we have $\left (\frac{\Delta}p\right )\left
(\frac{\Delta'}p\right )=-1$. In any case get
$$\left (\frac{\Delta}p\right )+\left (\frac{\Delta'}p\right )=0.$$
     \end{proof}

     We resume the proof of the theorem.

We put
    $$A = \bigcup_{u = t + 1}^{t + k} A_u. $$
    and we will show that for  a suitable choice of t this will be a good set.

Since $(0,0)\in A_u$  for all $u$ and the rest of the $A_u$'s are disjoint, we have
$|A| = k (p-1) + 1$.

     We can estimate the corresponding representation function as
$$r(x) \leq  \sum _{u,v=t+1}^{t+k} r_{u,v} (x)$$
     (equality fails sometimes, because representations involving $(0,0)$ are
counted once on the left and several times on the right).

     We parametrize the variables of summation as
$u = t + i, v = t + j$ with $1 \leq i, j \leq k$. So $2 \leq i + j \leq 2k$ and we can write
$i + j = k + 1 + l$ with $|l| \leq k - 1$.

     For fixed $l$, we have $k - |l|$ pairs $i, j$ (which means $k - |l|$ pairs
$u, v$). These pairs can be split into two groups: $n^+$ of them will have
$\left( \frac{uv}{p} \right) = 1$ and $n^-$ will have
$\left( \frac{uv}{p} \right) = -1$. Clearly
     \[ n^+ + n^- = k - |l|, \ n^+ - n^- = \sum \left(\frac{uv}{p} \right) .\]

     Of these $n^+ + n^-$ pairs we can combine
$\min\{n^+, n^-\}$ into pairs of pairs with opposite quadratic character, that
is,  with $\left(\frac{uvu'v'}{p} \right) = -1$. For these
we use Lemma \ref{u+v} to estimate the sum of the corresponding representation
functions $r_{u,v}+r_{u',v'}$ by 2. For the uncoupled pairs we can only
estimate the individual values by 2. Altogether this gives

\begin{eqnarray*}
\sum _{i + j= k + 1 + l} r_{u,v} (x) & \leq  & 2 (\min\{n^+, n^-\}) + 2(\max\{n^+, n^-\} -
\min\{n^+, n^-\}) \\ & = & 2(\max\{n^+, n^-\}) \\
 & = & n^+ + n^- + |n^+ - n^-| \\
& = & k - |l| + \left| \sum \left(\frac{uv}{p} \right) \right|. \\
\end{eqnarray*}

     Adding this for all possible value of $l$, for a fixed $t$ we obtain
     $$r(x) \leq k^2 + \sum_{|l| \leq k - 1} \left| \sum_{i + j  = k + 1 + l} \left( \frac{(t + i)(t
+ j)}{p} \right) \right| = k^2 + S_t . $$

     We are going to show that $S_t$ is small on average. Since we need values
with $u, v \not \equiv 0$, we can use only $0\leq t\leq p-1-k$; however, the complete sum is easier
to work with.
Applying the Cauchy-Schwarz
inequality we get
\begin{eqnarray*}
\sum_{t=0}^{p-1} S_t & = & \sum_{t,l} \left| \sum_{i + j = k + 1 + l} \left( \frac{(t + i)(t + j)}{p} \right) \right|\\
& \leq & \sqrt{2kp \sum_{l,t} \left( \sum_{i+j = k+1+l} \left( \frac{(t+i)(t+j)}{p} \right) \right)^2}\\
& \leq & \sqrt{2kp \sum_{i + j = i' + j'} \sum_t \left( \frac{(t+i)(t+j)(t+i')(t+j')}{p} \right)}.
\end{eqnarray*}

     To estimate the inner sum we use Weil's Theorem that asserts
     $$\left| \sum_{t=0}^{p-1} \left( \frac{f(t)}{p} \right) \right| \leq \deg f
\sqrt{p}$$
     for any polynomial $f$ which is not a constant multiple of a square. Hence
     $$ \sum_{t=0}^{p-1} \left( \frac{(t+i)(t+j)(t+i')(t+j')}{p} \right) \leq 4 \sqrt{p} $$
     except when the enumerator as a polynomial of $t$ is a square.

     The numerator will be a square if the four numbers $i,i',j,j'$ form two
equal pairs. This happens exactly $k(2k-1)$ times. Indeed, we may have $i=i'$,
$j=j'$, $k^2$ cases, or $i=j'$, $j=i'$, another $k^2$ cases. The $k$ cases when
all four coincide have been counted twice. Finally, if $i=j$ and $i'=j'$, then
the equality of sums implies that all are equal, so this gives no new case. In
these cases for the sum we use the trivial upper estimate $p$.

     The total number of quadruples $i,i',j,j'$ is $\leq k^3$, since three of them
determine the fourth uniquely.

     Combining our estimates we obtain
$$\sum_{t=0}^{p-1} S_t \leq \sqrt{2p^2k^2(2k-1)+ 8p^{3/2}k^4}. $$
     This implies that there is a value of $t$, $0\leq t\leq p-k-1$ such that
     $$ S_t \leq \frac{\sqrt{2p^2k^2(2k-1)+ 8p^{3/2}k^4}}{p-k} < 2k^{3/2} $$
     if $p$ is large enough. This yields that $r(x)<k^2 + 2k^{3/2}$ as claimed.
     \end{proof}

     \section{Construction in certain cyclic groups}

     In this section we show how to project a set from ${\mathbb {Z}}_p^2$ into ${\mathbb {Z}}_q$ with
$q=p^2s$.

     \begin{Th} 
     Let $A \ce \mathbb{Z}_p^2$ be a $g$-Sidon set with $|A| = m$, and put $q=p^2s$ with
a positive integer $s$. There is a $g'$-Sidon set $A' \ce {\mathbb
{Z}}_q$ with $|A'| = ms$  and $g'=g(s+1)$.
     \end{Th}

     \begin{proof}
     An element of $A$ is a pair of residues modulo $p$, which we shall
represent by integers in $[0, p-1]$. Given and element $(a,b)\in A$, we put into
$A'$ all numbers of the form
$a + cp + bsp$ with $0 \leq c \leq s-1$. Clearly $\ab {A'}=sm$.

     To estimate the representation function of $A'$ we need to tell,
given $a, b, c$, how many $a_1, b_1, c_1, a_2, b_2, c_2$ are there such that
     \begin{equation} \label{c1}
     a + cp + bsp \equiv a_1 + c_1p + b_1sp + a_2 + c_2 p + b_2sp  \pmod{ p^2s}\end{equation}
     with $ (a_1, b_1), (a_2, b_2) \in A$ and $0 \leq c_1, c_2 \leq s-1$.

     First consider congruence \eqref  {c1} modulo $p$. We have
     $$a \equiv a_1 + a_2 \pmod{ p}, $$
     hence $a_1 + a_2 = a + \delta p$  with $ \delta = 0 $ or 1.
     We substitute this into \eqref  {c1}, substract $a$ and divide by $p$ to
obtain
      $$c + bs \equiv \delta + c_1 + c_2 + (b_1 + b_2)s \pmod{ ps} . $$
      We take this modulo $s$:
$$c \equiv \delta + c_1 + c_2 \pmod{ s}, $$
     consequently $\delta + c_1 + c_2 = c + \eta s$  with $ \eta = 0 $ or 1.
     Again substituting back, substracting $c$ and dividing by $s$ we obtain
     $$b \equiv \eta + b_1 + b_2 \pmod{ p}.$$

So $(a,b) = (a_1, b_1) + (a_2, b_2) + (0,\eta)$ which means that for $a,
b, \eta$ given, we have $\leq g$  possible values of $a_1, b_1, a_2, b_2$.

    Now we are going to find the number of possible values of $c_1, c_2$
for $a, b, c, \eta,
a_1, b_1, a_2, b_2$ given.

     Observe that from these data we can calculate $\delta = (a_1 + a_2 - a)/p$.
For $c_1, c_2$ we have the equation
$c_1 + c_2 = c - \delta + \eta s$.

     If $\eta =0$, we have $c_1\leq c$, at most $c+1$ possibilities.

     If $\eta =1$, we have $c_1+c_2\geq c+s-1$, hence $c-1 < c_1\leq s-1$, which gives at
most $s-c$ possibilities.

     Hence, if $a,b,c,\eta $ are given, our estimate is $g(c+1)$ or $g(s-c)$,
depending on $\eta $. Adding the two estimates we get the claimed bound $g(s+1)$.
     \end{proof}

     On combining this result with Theorem \ref{pxp} we obtain the following
result.

\begin{Th} \label{ciclicp2s}
     For any positive integers $k,s$, for every sufficiently large prime $p$,
there is a set
$A \ce {\mathbb {Z}}_{p^2s}$ with $(kp - k + 1)s$ elements which is a $\lfloor k^2 + 2
k^{3/2} \rfloor(s+1)$-Sidon set.
     \end{Th}

 Put $q=p^2s$ and $g=\lfloor k^2 + 2
k^{3/2} \rfloor (s+1)$.  Thus,
\begin{eqnarray*}
\frac{\alpha_g(q)}{\sqrt{gq}} \ge  \frac{|A|}{\sqrt{gq}} & = &
\frac{(kp-k+1)s}{\sqrt{\lfloor k^2 + 2 k^{3/2} \rfloor (s+1)
p^2s}} \\
& \ge & \frac{(kp-k)s}{\sqrt{(k^2 + 2 k^{3/2})(s+1) p^2s}}\\
& \ge & \frac{p-1}{p\sqrt{(1 + 2/\sqrt k)(1+1/s)}}.
\end{eqnarray*}

A convenient choice of the parameters is  $k=4s^2$ (so $s = \Theta(g^{1/5})$). Assuming that, we get
\begin{equation*}
\frac{\alpha_g(q)}{\sqrt{gq}}\ge \frac{p-1}p\cdot \frac 1{1+1/s}.
\end{equation*}
Thus, the Prime Number Theorem says that
\begin{equation*}
\frac{\alpha_g}{\sqrt g}=\limsup_{q\to
\infty}\frac{\alpha_g(q)}{\sqrt{gq}}\ge \limsup_{p\to \infty
}\frac{p-1}p\cdot \frac 1{1+1/s}=1+O(g^{-1/5}),
\end{equation*}
which completes the proof of Theorem \ref{modular}.


\section{Upper bound}

We turn now to the proof of Theorem \ref{integer}, which says:
$$\lim_{g \rightarrow \infty} \liminf_{N \rightarrow \infty} \frac{\beta_g(N)}{\sqrt{g}{\sqrt{N}}} = \lim_{g \rightarrow \infty} \limsup_{N \rightarrow \infty} \frac{\beta_g(N)}{\sqrt{g}{\sqrt{N}}} = \sigma.$$

We will prove it in two stages:
\begin{enumerate}
\item [Part A.] $$\limsup_{g \rightarrow \infty} \limsup_{N \rightarrow \infty} \frac{\beta_g(N)}{\sqrt{g}{\sqrt{N}}} \le \sigma.$$
\item [Part B.] $$\liminf_{g \rightarrow \infty} \liminf_{N \rightarrow \infty} \frac{\beta_g(N)}{\sqrt{g}{\sqrt{N}}} \ge \sigma.$$
\end{enumerate}

For Part A we will use the ideas of Schinzel and Schmidt
\cite{Schin}, which give a connection between convolutions and number
of representations, between the continuous and the discrete world.
For the sake of completeness we rewrite the results and the proofs
in a more convenient way for our purposes.

\bigskip

Remember from (\ref{sigma}) the definition of $\sigma$:
$$\sigma = \sup_{f \in \mathcal{F}} |f|_1,$$
where $\mathcal{F} = \{f: f \ge 0, \ \text{supp}(f) \ce [0,1], \ |f*f|_{\infty} \le 1 \}.$

\smallskip

We will use the next result, which is assertion (ii) of Theorem 1 in
\cite{Schin} (essentially the same result appears in \cite{Mar2} as Corollary 1.5):
\begin{Th} \label{PolSchin}
Let $\sigma$ be the constant defined above and $\mathcal{Q}_N = \{Q \in \mathbb{R}_{\ge 0} [x] : Q \not \equiv 0, \deg Q < N\}$. Then
$$\sup_{Q \in \mathcal{Q}_N} \frac{|Q|_1}{\sqrt{N} \sqrt{|Q^2|_{\infty}}} \le \sigma,$$
where $|P|_1$ is the sum and $|P|_{\infty}$ the maximum of the coefficients of a polynomial $P$.
\end{Th}

\begin{proof}
First of all, observe that the definition of $\sigma$ is equivalent to this one:
$$\sigma = \sup_{g \in \mathcal{G}} \frac{|g|_1}{\sqrt{|g*g|_{\infty}}},$$
where $\mathcal{G} = \{g: g \ge 0, \ \text{supp}(g) \ce [0,1] \}.$

\smallskip

Given a polynomial $Q = a_0 + a_1 x + \ldots + a_{N-1} x^{N-1}$ in $\mathcal{Q}_N$, we define the step function $g$ with support in $[0,1)$ having
$$g(x) = a_i \ \text{ for } \ \frac{i}{N} \le x < \frac{i+1}{N} \ \text{ for every } \ i = 0, 1, \ldots, N-1.$$

The convolution of this step function with itself is the polygonal function:
$$ g*g(x) = \sum_{i=0}^{j}a_i a_{j-i} \left(x - \frac{j}{N} \right) + \sum_{i=0}^{j-1}a_i a_{j-1-i} \left(\frac{j+1}{N} - x \right) \text{ if } x \in \left[ \frac{j}{N}, \frac{j+1}{N} \right)$$
for every $j=0, 1, \ldots, 2N - 1$, where we define $a_N = a_{N+1} = \ldots = a_{2N-1} = 0$.

So, $$\sup_{x} (g*g)(x) = \frac{1}{N} \sup_{0 \le j \le 2N-2} \left( \sum_{i=0}^{j} a_i a_{j-i} \right).$$

Since, obviously, $\int_0^1 g(x) \ dx = \frac{1}{N} \sum_{i=0}^{N-1} a_i$, we have:
$$\frac{|Q|_1}{\sqrt{N} \sqrt{|Q^2|_{\infty}}} = \frac{\int_0^1 g(x) \ dx}{\sqrt{\sup_{x} (g*g)(x)}} \le \sigma.$$

And because we have this for every $Q$, the theorem is proved.
\end{proof}

Now, given a $g$-Sidon set $A \ \ce \ \{0, 1, \ldots, N - 1 \}$, we define the polynomial $Q_A(x) = \sum_{a \in A} x^a$, so $Q_A^2(x) = \sum_{n} r(n) x^n$. Then, Theorem \ref{PolSchin} says that
$$\sigma \ge \frac{|Q_A|_1}{\sqrt{|Q_A^2|_{\infty}} \sqrt{N}} \ge \frac{|A|}{\sqrt{g} \sqrt{N}}.$$

Since this happens for every $g$-Sidon set in $\{0, 1, \ldots, N-1\}$, we have that
$$\frac{\beta_g(N)}{\sqrt{g}{\sqrt{N}}} \le \sigma.$$
This proves Part A of Theorem \ref{integer}, which is the easy
part.

\begin{Rem}
In fact, not only Schinzel and Schmidt prove the result above in \cite{Schin}, but they also prove (see Theorem \ref{SchSch}) that
$$\lim_{N \to \infty} \sup_{Q \in \mathcal{Q}_N} \frac{|Q|_1}{\sqrt{N} \sqrt{|Q^2|_{\infty}}} = \sigma.$$

Newman polynomials are polynomials all of whose coefficients are 0 or 1. In \cite{Yu}, Gang Yu conjectured that for every sequence of Newman polynomials $Q_N$ with $\deg Q_N = N-1$ and $|Q_N|_1 = o(N)$
$$\limsup_{N \to \infty} \frac{|Q_N|_1}{\sqrt{N} \sqrt{|Q_N^2|_{\infty}}} \le 1.$$

Greg Martin and Kevin O'Bryant \cite{Mar2} disproved this conjecture, finding a sequence of Newman polynomials with $\deg Q_N = N-1$, $|Q_N|_1 = o(N)$ and $$\limsup_{N \to \infty} \frac{|Q_N|_1}{\sqrt{N} \sqrt{|Q_N^2|_{\infty}}} = \frac{2}{\sqrt{\pi}}.$$

In fact, with the probabilistic method it can be proved without much effort that there is a sequence of Newman polynomials, with $\deg Q_N = N-1$ and $|Q_N|_1 = O(N^{1/2} (\log N)^{\beta})$ for any given $\beta > 1/2$, such that
$$\limsup_{N \to \infty} \frac{|Q_N|_1}{\sqrt{N} \sqrt{|Q_N^2|_{\infty}}} = \sigma.$$

Our Theorem \ref{integer} says that given $\varepsilon > 0$, there exists a constant $c_\varepsilon$ and a sequence of polynomials, $Q_N$, with $\deg Q_N = N-1$ and $|Q_N|_1 \le c_{\varepsilon} N^{1/2}$ such that
$$\limsup_{N \to \infty} \frac{|Q_N|_1}{\sqrt{N} \sqrt{|Q_N^2|_{\infty}}} \ge \sigma - \varepsilon.$$

Observe that this growth is close to the best possible, since taking $|Q_N|_1 = o(N^{1/2})$ makes $\frac{|Q_N|_1}{\sqrt{N} \sqrt{|Q_N^2|_{\infty}}} \to 0$.
\end{Rem}

\section{Connecting the discrete and the continuous world}

For Part B of the proof of Theorem \ref{integer} we will need
another result of Schinzel and Schmidt (assertion (iii) of Theorem 1
in \cite{Schin}) which we state in a more convenient form for our
purposes:

\begin{Th} \label{SchSch}
For every $0 < \alpha < 1/2$, for any $0 < \varepsilon < 1$ and for every $n > n(\varepsilon)$, there exist non-negative real numbers $a_0, a_1, \ldots, a_n$ such that
\begin {enumerate}
 \item $a_i \le n^{\alpha} (1 - \varepsilon)$ for every $i = 0, 1, \ldots, n$.
 \item $\sum_{i=0}^{n} a_i \ge n \sigma (1 - \varepsilon)$.
 \item $\sum_{0 \le i, m-i \le n} a_i a_{m-i} \le n (1 + \varepsilon)$ for every $m = 0, 1, \ldots, 2n$.
\end {enumerate}
\end{Th}


\begin{proof}
We start with a real nonnegative function defined in $[0,1]$, $g$, with $|g * g|_{\infty} \le 1$ and $|g|_1$ close to $\sigma$, say $|g|_1 \ge \sigma (1 - \varepsilon/2)$.

\bigskip

For $r < s$ we have the estimation
\begin{eqnarray} \label{estim1}
\left( \int_{r}^{s} g(x) \ dx \right)^2 & = & \int_{r}^{s} \int_{r}^{s} g(x) g(y) \ dx \ dy \nonumber \\
& = & \int_{r + x}^{s + x} \int_{r}^{s} g(x) g(z - x) \ dx \ dz \\
& \le & \int_{2r}^{2s} \int_{r}^{s} g(x) g(z - x) \ dx \ dz \le 2 (s - r) \nonumber
\end{eqnarray}

which implies that
\begin{equation} \label{estim}
\int_{r}^{s} g(x) \ dx \le \sqrt{2(s-r)}.
\end{equation}

\bigskip

Trying to ``discretize'' our function $g$, we define for $i = 0, 1, 2, \ldots, n$:
$$a_i = \frac{n}{2L} \int_{(i - L)/n}^{(i + L)/n} g(x) \ dx$$
where $1 \le L \le n/2$ is an integer that will be determined later.

\bigskip

Estimation (\ref{estim}) proves that
\begin {equation} \label{primera}
a_i \le \sqrt{n/L} \ \text{ for } \ i = 0, 1, 2, \ldots, n.
\end {equation}

\bigskip

Now we give a lower bound for the sum $\sum_{i=0}^{n} a_i$:
$$\sum_{i=0}^{n} a_i = \frac{n}{2L} \int_{0}^{1} \nu(x) g(x) \ dx,$$
where
\begin{eqnarray*}
\nu (x) & = & \sharp \left\{ i \in [0,n] : \frac{i-L}{n} \le x \le \frac{i+L}{n} \right\} \\
& = & \sharp \left\{ i : \max\{ 0, nx - L \} \le i \le \min \{ n, nx + L \} \right\}.
\end{eqnarray*}

\smallskip

Taking in account that an interval of length $M$ has $\ge \lfloor M \rfloor$ integers and an interval of length $M$ starting or finishing at an integer has $\lceil M \rceil$ integers, and since $L \in \mathbb Z$ and $1 \le L \le n/2$, we have
\begin{displaymath}
\nu (x) \ge
\left\{
\begin{array}{lll}
nx + L = 2L - (L - nx) & \textrm{ if } 0 \le x \le L/n \\
2L & \textrm{ if } L/n \le x \le 1 - L/n \\
n - nx + L = 2L - (L - n(1-x)) & \textrm{ if } 1 - L/n \le x \le 1 \\
\end{array}
\right.
\end{displaymath}
and so
$$\sum_{i=0}^{n} a_i \ge n \int_{0}^{1} g(x) \ dx - \frac{n}{2L} \int_{0}^{L/n} (L - nx) g(x) \ dx - \frac{n}{2L} \int_{1 - L/n}^{1} (L - n(1-x)) g(x) \ dx.$$

Now, using the fact that $|g|_1 \ge \sigma (1 - \varepsilon / 2)$ and estimation (\ref{estim}),
\begin{equation} \label{segunda}
\sum_{i=0}^{n} a_i \ge n \sigma (1 - \varepsilon / 2) - \sqrt{2nL}.
\end{equation}

\bigskip

Also, for every $m \le 2n$ we give an upper bound for the sum $\sum_{0 \le i, m-i \le n} a_i a_{m-i}$. First we write:

$$\sum_{0 \le i, m-i \le n} a_i a_{m-i} = \left( \frac{n}{2L} \right)^2 \sum_{0 \le i, m-i \le n} \int_{(m - i - L)/n}^{(m - i + L)/n} \int_{(i - L)/n}^{(i + L)/n} g(x) g(y) \ dx \ dy.$$

\smallskip

Now, as in (\ref{estim1}), we set $z = x + y$ and we consider the set:
$$S_i = \left\{(x, z) : \frac{i - L}{n} \le x \le \frac{i+L}{n} \ \text{ and } \ \frac{m - i - L}{n} \le z - x \le \frac{m - i + L}{n} \right\}.$$

Then,
$$\sum_{0 \le i, m-i \le n} a_i a_{m-i} = \left( \frac{n}{2L} \right)^2 \sum_{0 \le i, m-i \le n} \int \int_{S_i} g(x) g(z - x) \ dx \ dz$$
and, defining $\mu(x,z) = \sharp \{\max\{0, m-n\} \le i \le \min\{m, n\} : i - L \le nx \le i + L \ \text{  and  } \ m - i - L \le n(z - x) \le m - i + L \}$,
$$\sum_{0 \le i, m-i \le n} a_i a_{m-i} = \left( \frac{n}{2L} \right)^2 \int \int \mu(x, z) g(x) g(z - x) \ dx \ dz.$$
If we write $h = i - nx$ then we are imposing $-L \le h \le L$ and $m - L - nz \le h \le m + L - nz$, so $$-L + \max \{0, m - nz \} \le h \le L + \min \{0, m - nz \},$$
and $\mu (x,z) \le \lambda(z)$, which is the number of $h$'s in this interval (it could be empty), and this number is clearly $\le 2L + 1$. Also, for each fixed $h$, $z$ moves in an interval of length $2L/n$.

This means (remember that $|g * g|_{\infty} \le 1$)
\begin{eqnarray*}
\sum_{0 \le i, m-i \le n} a_i a_{m-i} & \le & \left( \frac{n}{2L} \right)^2 \int \lambda (z) \int g(x) g(z - x) \ dx \ dz \\
& \le & \left( \frac{n}{2L} \right)^2 \int \lambda (z) \ dz \\
& \le & \left( \frac{n}{2L} \right)^2  \frac{2L(2L + 1)}{n}
\end{eqnarray*}
so the sum
\begin{equation} \label{tercera}
\sum_{0 \le i, m-i \le n} a_i a_{m-i} \le n \left( 1 + \frac{1}{2L} \right).
\end{equation}

Finally, looking at (\ref{primera}), (\ref{segunda}) and (\ref{tercera}), and choosing the integer $L = \lceil n^{1 - 2 \alpha} / (1 - \varepsilon)^2 \rceil$ with $0 < \alpha < 1/2$, for sufficiently large $n$ we'll have:
$$a_i \le n^{\alpha} (1 - \varepsilon) \qquad , \qquad \sum_{i=0}^{n} a_i \ge n \sigma (1 - \varepsilon) \qquad \text{and} \qquad \sum_{0 \le i, m-i \le n} a_i a_{m-i} \le n ( 1 + \varepsilon).$$
\end{proof}

\begin{Rem}
Now, we will construct random sets. We want to use the numbers obtained in Theorem \ref{SchSch} to define probabilities, $p_i$, and it will be convenient to know the sum of the $p_i$'s. This is the motivation for defining
$$p_i = a_i \cdot \dfrac{\sigma n^{1 - \alpha}}{\sum_{i=0}^{n} a_i} \quad \text{ for } \quad i = 0, 1, \ldots, n.$$
\end{Rem}

\smallskip

Now we fix $\alpha=1/3$, although any $\alpha \in (0,1/2)$ would
work. Then we have $p_i = a_i \cdot \dfrac{\sigma
n^{2/3}}{\sum_{i=0}^{n} a_i}$, so for any $0 < \varepsilon < 1$ and
for every $n > n(\varepsilon)$, we have $p_0, p_1, \ldots, p_n$ such
that:
$$p_i \le 1 \qquad , \qquad \sum_{i=0}^{n} p_i = \sigma n^{2/3} \qquad \text{ and } \qquad  \sum_{0 \le i, m-i \le n} p_i p_{m-i} \le n^{1/3} \dfrac{1 + \varepsilon}{(1 - \varepsilon)^2}.$$

\bigskip

In order to prove that the number of elements and the number of representations in our probabilistic sets are what we expect with high probability, we'll use Chernoff's inequality (see Corollary 1.9 in \cite{Tao}).

\begin{Prop} \label{Cher}
\textbf{(Chernoff's inequality)} Let $X=t_1+\cdots +t_n$ where the $t_i$ are independent Boolean random variables. Then for
any $\delta>0$
\begin{eqnarray}
\PP(|X-\E(X)|\ge \delta\E(X)) \le 2e^{-\min(\delta^2/4,\delta/2)\E(X)}.
\end{eqnarray}
\end{Prop}

\smallskip

Then, we have the next two lemmas which also appear in \cite{CV}:

\begin{Lemma} \label{number}
We consider the probability space of all the subsets $A \ \ce \ \{0, 1, \ldots, n\}$ defined by $\PP(i \in A) = p_i$.
With the $p_i$'s defined above, given $0 < \varepsilon < 1$, there exists $n_0(\varepsilon)$ such that, for all $n \ge n_0$,
$$\PP(|A| \ge \sigma n^{2/3} (1 - \varepsilon)) > 0.9.$$
\end{Lemma}

\begin{proof}
Since $|A|$ is a sum of independent Boolean variables and $\E(|A|) = \sum_{i=0}^{n} p_i = \sigma n^{2/3}$, we can apply Proposition \ref{Cher} to deduce that for large enough n
$$\PP(|A| < \sigma n^{2/3} (1 - \varepsilon) ) \le 2 e^{- \sigma n^{2/3} \varepsilon^2 / 4} < 0.1.$$
\end{proof}

\begin{Lemma} \label{repre}
We consider the probability space of all the subsets $A \ \ce \ \{0, 1, \ldots, n\}$ defined by $\PP(i \in A) = p_i$.
Again for the $p_i$'s defined above, given $0 < \varepsilon < 1$, there exists $n_1(\varepsilon)$ such that, for all $n \ge n_1$,
$$r (m) \le n^{1/3} \left( \frac{1 + \varepsilon}{1 - \varepsilon} \right)^3 \text{ for all } m = 0, 1, \ldots, 2n$$
with probability $> 0.9$.
\end{Lemma}

\begin{proof}
Since $r(m) = \sum_{0 \le i, m-i \le n} \I(i \in A) \I(m - i \in A)$ is a sum of Boolean variables which are not independent, it is convenient to consider
$$r'(m)/2 = \sum_{\substack{0 \le i, m-i \le n \\ i < m/2}} \I(i \in A) \I(m - i \in A)$$
leaving in mind that $r(m) = r'(m) + \I(m/2 \in A)$.

\smallskip

From the independence of the indicator functions, and following the notation introduced in Definition \ref{defrep}, the expected value of $r'(m)/2$ is
\begin{eqnarray*}
\mu_m & = & \E(r'(m)/2) = \sum_{\substack{0 \le i, m-i \le n \\ i < m/2}} \E(\I(i \in A)\I(m - i \in A)) \\
& = & \sum_{\substack{0 \le i, m-i \le n \\ i < m/2}} p_i p_{m-i} \le \frac{n^{1/3}}{2} \cdot \frac{1 + \varepsilon}{(1 - \varepsilon)^2}, 
\end{eqnarray*}
for every $m = 0, 1, \ldots, 2n$, for $n$ large enough.

\smallskip

If $\mu_m = 0$ then $\PP(r'(m)/2 > 0) = 0$, so we can consider the next two cases:
\begin{itemize}
\item If $\frac{1}{3} \cdot \frac{n^{1/3}(1 + \varepsilon)}{2 (1 - \varepsilon)^2} \le \mu_m$, we can apply Proposition \ref{Cher} (observe that $\varepsilon < 2$ and then $\varepsilon^2/4 \le \varepsilon/2$) to have
\begin{eqnarray*}
\PP \left( r'(m)/2 \ge \frac{n^{1/3}}{2} \left( \frac{1 + \varepsilon}{1 - \varepsilon} \right)^2 \right) & \le & \PP ( r'(m)/2 \ge \mu_m (1 + \varepsilon)) \\
& \le & 2 \exp \left(- \frac{\varepsilon ^2 \mu_m}{4} \right) \\
& \le & 2 \exp \left(- \frac{n^{1/3} \varepsilon ^2 (1 + \varepsilon)}{24 (1 - \varepsilon)^2} \right)
\end{eqnarray*}

\item If $0 < \mu_m < \frac{1}{3} \cdot \frac{n^{1/3}(1 + \varepsilon)}{2 (1 - \varepsilon)^2}$ then we define $\delta = \frac{n^{1/3}}{2 \mu_m} \left( \frac{1 + \varepsilon}{1 - \varepsilon} \right)^2 - 1$ (observe that $\delta \ge 2$ and then $\delta/2 \le \delta^2/4$) and we can apply Proposition \ref{Cher} to have
\begin{eqnarray*}
\PP \left( r'(m)/2 \ge \frac{n^{1/3}}{2} \left( \frac{1 + \varepsilon}{1 - \varepsilon} \right)^2 \right) & = &  \PP(r'(m)/2 \ge \mu_m (1 + \delta)) \\
& \le & 2 \exp \left(- \frac{\delta \mu_m}{2} \right) \\
& = & 2 \exp \left( - \frac{n^{1/3}}{4} \left( \frac{1 + \varepsilon}{1 - \varepsilon} \right)^2 + \frac{\mu_m}{2} \right) \\
& \le & 2 \exp \left( - \frac{n^{1/3}}{4} \left( \frac{1 + \varepsilon}{1 - \varepsilon} \right)^2 + \frac{n^{1/3}(1 + \varepsilon)}{12(1 - \varepsilon)^2} \right) \\
& \le & 2 \exp \left( - \frac{n^{1/3}}{6} \left( \frac{1 + \varepsilon}{1 - \varepsilon} \right)^2 \right) \\
\end{eqnarray*}
\end{itemize}

Then,
\begin{eqnarray*}
&& \PP \left( r'(m)/2 \ge \frac{n^{1/3}}{2} \left( \frac{1 + \varepsilon}{1 - \varepsilon} \right)^2 \text{ for some } m \right)\\
&& \le 4n \left( \exp \left(- \frac{n^{1/3} \varepsilon ^2 (1 + \varepsilon)}{24 (1 - \varepsilon)^2} \right) + \exp \left( - \frac{n^{1/3}}{6} \left( \frac{1 + \varepsilon}{1 - \varepsilon} \right)^2 \right) \right)
\end{eqnarray*}
which is $< 0.1$ for $n$ large enough.

\smallskip

Remembering that $r(m) = r'(m) + \I(m/2 \in A)$,
$$\PP \left( r(m) \ge n^{1/3} \left( \frac{1 + \varepsilon}{1 - \varepsilon} \right)^2 + \ \I(m/2 \in A) \ \text{ for some } m \right) < 0.1 \text{ for } n \text{ large enough},$$
and finally
$$\PP \left( r(m) \ge n^{1/3} \left( \frac{1 + \varepsilon}{1 - \varepsilon} \right)^3 \text{ for some } m \right) < 0.1 \text{ for } n \text{ large enough}.$$

\end{proof}

Lemmas \ref{number} and \ref{repre} imply that, given $0 < \varepsilon < 1$, for $n \ge \max \{n_0, n_1\}$, the probability that our random set $A$ satisfies $|A| \ge \sigma n^{2/3} (1 - \varepsilon)$ and $r(m) \le n^{1/3} \left( \frac{1 + \varepsilon}{1 - \varepsilon} \right)^3$ for every $m$ is greater than $0.8$. In particular, for every $n \ge \max\{n_0, n_1\}$ we have a set $A \ce \{0, 1, \ldots, n\}$ satisfying these conditions.

\section{From residues to integers}

In order to prove Part B of Theorem \ref{integer}, we will also need the next lemma, which allows us to ``paste'' copies of a $g_2$-Sidon set in a cyclic group with a dilation of a $g_1$-Sidon set in the integers.

\begin{Lemma} \label{pasting}
Let $A = \{0 = a_1 < \ldots < a_k \}$ be a $g_1$-Sidon set in $\mathbb{Z}$ and let $C \subseteq [1, q]$ be a $g_2$-Sidon set $\pmod{q}$. Then $B = \cup_{i=1}^{k}(C + q a_i)$ is a $g_1 g_2$-Sidon set in $[1, q(a_k + 1)]$ with $k|C|$ elements.
\end{Lemma}

\begin{proof}
Suppose we have $g_1g_2 + 1$ representations of an element as the sum of two
$$b_{1,1} + b_{2,1} = b_{1,2} + b_{2,2} = \cdots = b_{1,g_1 g_2 + 1} + b_{2,g_1 g_2 + 1}.$$

Each $b_{i,j} = c_{i,j} + q a_{i,j}$ in a unique way. Now we can look at the equality modulo $q$ to have
$$c_{1,1} + c_{2,1} = c_{1,2} + c_{2,2} = \cdots = c_{1,g_1 g_2 + 1} + c_{2,g_1 g_2 + 1} \pmod{q}.$$

Since $C$ is a $g_2$-Sidon set $\pmod{q}$, by the pigeonhole principle, there are at least $g_1 + 1$ pairs $(c_{1,i_1},c_{2,i_1})$, ..., $(c_{1,i_{g_1 + 1}},c_{2,i_{g_1 + 1}})$ such that:
$$c_{1,i_1} = \cdots = c_{1,i_{g_1 + 1}} \qquad \text{and} \qquad c_{2,i_1} = \cdots = c_{2,i_{g_1 + 1}}.$$

So the corresponding $a_i$'s satisfy
$$a_{1,i_1} + a_{2,i_1} = \cdots = a_{1,i_{g_1 + 1}} + a_{2,i_{g_1 + 1}},$$
and since $A$ is a $g_1$-Sidon set, there must be an equality
$$a_{1, k} = a_{1, l} \qquad \text{and} \qquad a_{2, k} = a_{2, l}$$
for some $k, l \in \{ i_1, \ldots, i_{g_1 + 1} \}$.

Then, for these $k$ and $l$ we have
$$b_{1, k} = b_{1, l} \qquad \text{and} \qquad b_{2, k} = b_{2, l},$$
which completes the proof.
\end{proof}

\bigskip

With all these weapons, we are ready to finish our proof.

\bigskip

Given $0 < \varepsilon < 1$ we have that:
\begin{enumerate}
\item [a)] For every large enough $g$ we can define $n = n(g)$ as the least integer such that $g = \lfloor n^{1/3} \left( \frac{1 + \varepsilon}{1 - \varepsilon} \right)^3 \rfloor$, and such an $n$ exists because $n^{1/3} \left( \frac{1 + \varepsilon}{1 - \varepsilon} \right)^3$ grows more slowly than $n$. Observe that $n(g) \to \infty$ when $g \to \infty$.

Now, by lemmas \ref{number} and \ref{repre}, there is $g_0 = g_0(\varepsilon)$ such that for every $g_1 \ge g_0$ we can consider $n = n(g_1)$ and we have a $g_1$-Sidon set $A \ce \{0, 1, \ldots, n\}$ such that
$$\frac{|A|}{\sqrt{g_1}\sqrt{n+1}} \ge \sigma \sqrt{\frac{n}{n + 1}} \cdot \frac{(1 - \varepsilon)^{5/2}}{(1 + \varepsilon)^{3/2}}.$$

\item [b)] By Theorem \ref{ciclicp2s}, there are $g_2 = g_2(\varepsilon)$, $s = s(\varepsilon)$ and a sequence $q_0 = p_r^2 s$, $q_1 = p_{r+1}^2 s$, $q_2 = p_{r+2}^2 s$, $\ldots$ (where $p_i$ is the $i$-th prime and $r = r(\varepsilon)$) such that for every $i = 0, 1, 2, \ldots$ there is a $g_2$-Sidon set $A_i \ce \mathbb{Z}_{q_i}$ with
$$\frac{|A_i|}{\sqrt{g_2 q_i}} \ge 1 - \varepsilon.$$
\end{enumerate}

\smallskip

So, given $0 < \varepsilon < 1$:
\begin{enumerate}
\item [1) ] For every $g \ge g_0(\varepsilon) g_2(\varepsilon)$ there is a $g_1 = g_1(g)$ such that
$$g_1 g_2 \le g < (g_1 + 1) g_2,$$
and we have $n = n(g_1)$ with $g_1 = \lfloor n^{1/3} \left( \frac{1 - \varepsilon}{1 + \varepsilon}\right)^3\rfloor$ and a $g_1$-Sidon set $A \ce \{0, 1, \ldots, n \}$ with
$$\frac{|A|}{\sqrt{g_1}\sqrt{n + 1}} \ge \sigma \sqrt{\frac{n}{n + 1}} \cdot \frac{(1 - \varepsilon)^{5/2}}{(1 + \varepsilon)^{3/2}}.$$

\smallskip

\item [2) ] For any $N \ge (n + 1) q_0$, there is an $i = i(N)$ such that
$$(n + 1) q_i \le N < (n + 1) q_{i + 1},$$
and we have a $g_2$-Sidon set $\pmod{q_i}$, $A_i$, with
$$\frac{|A_i|}{\sqrt{g_2 q_i}} \ge 1 - \varepsilon.$$

\end{enumerate}

Then, for any $g$ and $N$ large enough, applying Lemma \ref{pasting} we can construct a $g_1 g_2$-Sidon set from $A$ and $A_i$ with $|A||A_i|$ elements in $[1, N]$.

So we have that $\beta_g(N) \ge \beta_{g_1 g_2}(N) \ge |A||A_i|$ and then
\begin{eqnarray*}
\frac{\beta_g(N)}{\sqrt{g} \sqrt{N}} & \ge & \frac{\beta_{g_1 g_2}(N)}{\sqrt{(g_1 + 1)g_2} \sqrt{(n + 1)q_{i+1}}} \\
& \ge & \frac{|A| |A_i|}{\sqrt{g_1 g_2} \sqrt{(n + 1)q_i}} \sqrt{\frac{g_1}{g_1 + 1}} \sqrt{\frac{q_i}{q_{i+1}}} \\
& \ge & \sigma \frac{(1 - \varepsilon)^{7/2}}{(1 + \varepsilon)^{3/2}} \sqrt{\frac{n}{n + 1}} \sqrt{\frac{g_1}{g_1 + 1}} \sqrt{\frac{p_{r+i}}{p_{r+i+1}}}.
\end{eqnarray*}

Finally, as a consequence of the Prime Number Theorem, this means that, given $0 < \varepsilon < 1$, for $g$ and $N$ large enough
$$\frac{\beta_g(N)}{\sqrt{g} \sqrt{N}} \ge \sigma \frac{(1 - \varepsilon)^{9/2}}{(1 + \varepsilon)^{3/2}}$$
i. e.
$$\liminf_{g \rightarrow \infty} \liminf_{N \rightarrow \infty} \frac{\beta_g(N)}{\sqrt{g} \sqrt{N}} \ge \sigma.$$

\end{document}